\documentclass[10pt,letterpaper]{amsart}
\usepackage{latexsym, amscd, amsfonts, eucal, mathrsfs, amsmath, amssymb, amsthm, tikz-cd, xr, stmaryrd, amsxtra, mathtools, MnSymbol,bbm,mathrsfs}
\usepackage[alphabetic]{amsrefs}
\usepackage{enumitem}
\usepackage[hidelinks]{hyperref}
\usepackage{chngcntr}

\counterwithin{equation}{section}
\newtheorem{theorem}[equation]{Theorem}
\newtheorem{lemma}[equation]{Lemma}

\newtheorem{proposition}[equation]{Proposition}
\newtheorem{corollary}[equation]{Corollary}

\theoremstyle{definition}

\newtheorem{prop-con}[equation]{Proposition--Construction}

\newtheorem*{observation}{Observation}

\newtheorem{definition}[equation]{Definition}
\newtheorem{definition/lemma}[equation]{Definition--Lemma}
\newtheorem{example}[equation]{Example}

\theoremstyle{remark}

\title[Irreducibility of crystalline loci]{Irreducibility of some crystalline loci with irregular Hodge--Tate weights}

\author[R. Bartlett]{Robin Bartlett}
\address{University of Glasgow, United Kingdom}
\email{robin.bartlett.math@gmail.com}
\begin{document}
	
		\begin{abstract}
		We show that loci of crystalline representations of $G_K$ for $K/\mathbb{Q}_p$ an unramified extension are irreducible when the Hodge--Tate weights are fixed and sufficiently small. This was previously known for weights in the interval $[-p,0]$ and in this paper we show how that this bound can be relaxed provided the Hodge--Tate weights are sufficiently irregular at certain embeddings. This is motivated by the desire to extend the conjectures of Breuil--M\'ezard on loci of potentially crystalline representations to irregular weights.
	\end{abstract}
	\maketitle
	\section{Introduction}
	
	In this paper we demonstrate how techniques used to study moduli spaces of crystalline $p$-adic Galois representations with small Hodge--Tate weights relative to $p$ also apply to weights in much larger ranges whenever the weights are suitably irregular. For example, we prove:
	
	\begin{theorem}[see Corollary~\ref{cor-irred}]\label{thm-1.1}
		Let $K/\mathbb{Q}_p$ be a finite unramified  extension and let $\mathcal{X}_d^\mu$ denote the $p$-adic algebraic stack over $\operatorname{Spf}\overline{\mathbb{Z}}_p$ classifying $d$-dimensional crystalline representations of $\operatorname{Gal}(\overline{K}/K) = G_K$ with Hodge type $\mu$. For each embedding $\tau:K \hookrightarrow \overline{\mathbb{Q}}_p$ assume there are integers $N_\tau \geq 0$ so that 
		\begin{itemize}
			\item The $\tau$-th part of $\mu$ is contained in the interval $[-p^{N_\tau+1},0]$
			\item The $\tau \circ \varphi^n$-th part of $\mu$ is $0$ whenever $1 \leq n \leq N_\tau$ (here $\varphi$ denotes the automorphism of $K$ lifting the $p$-th power map on its residue field)
		\end{itemize}
		If $N_\tau=0$ for all $\tau$ assume also that $\mu$ is contained in $[-p+1,0]$ for at least one $\tau$.\footnote{This is to avoid the Steinberg weight where irreducibility is known to fail.} Then $\mathcal{X}_d^\mu \otimes_{\overline{\mathbb{Z}}_p} \overline{\mathbb{F}}_p$ is irreducible.
	\end{theorem}

	When all $N_\tau =0$ the theorem considers Hodge--Tate weights contained in $[-p,0]$. In this setting the claimed irreducibility was already known (see for example \cites{BBHKLLSW,B21,B23}) and the range $[-p,0]$ is sharp. The new content in this paper is that if the Hodge--Tate weights are zero at embeddings $\tau \circ \varphi^n,\tau\circ \varphi^{n-1},\ldots, \tau \circ \varphi$-th then the geometry of $\mathcal{X}_d^\mu$ behaves as it does in the $[-p,0]$ range whenever the Hodge--Tate weights at $\tau$ are in the larger range $[-p^{n+1},0]$. In fact, we show a stronger statement, namely that certain moduli spaces of crystalline Breuil--Kisin modules with Hodge--Tate weights in these ranges are smooth and irreducible (see Theorem~\ref{thm-main}).
	
	The motivation for Theorem~\ref{thm-1.1} comes from the desire to understand versions of the Breuil--M\'ezard conjecture for irregular weights, and Theorem~\ref{thm-1.1} can be viewed as giving some insight into the shape of such a conjecture. To explain this recall that the Breuil--M\'ezard conjecture \cite{BM02} describes (at the level of cycles) the mod $p$ loci of potentially crystalline representations with regular (i.e. consisting of distinct integers) Hodge--Tate weights. This is done in terms of the representation theory of $\operatorname{GL}_n(k)$ (for $k$ the residue field of $K$) and a finite collection of ``explicit'' cycles $Z_\sigma$ indexed by the irreducible representations of $\operatorname{GL}_n(k)$. When the highest weight of $\sigma$ is in the lowest alcove (and is non-Steinberg) then the conjectures identify $Z_\sigma$ with the cycle of a suitable $\mathcal{X}_d^\mu \otimes_{\mathcal{O}} \mathbb{F}$ with $\mu$ contained in the range $[-p,0]$. As a result one deduces irreducibility of such $Z_\sigma$ (and when $d=2$ these exhaust all the $Z_\sigma$'s besides the Steinberg one).
	
	Currently the shape of a Breuil--M\'ezard conjecture in irregular weight is completely unknown. However, one could imagine a similar formulation in which the mod $p$ loci of potentially crystalline representations with irregular Hodge--Tate weights are described in terms of a finite collection of cycles (necessarily of smaller dimension than the $Z_\sigma$) and some, currently mysterious, representation theoretic input. This viewpoint is corroborated by some recent calculations in \cite{BBHKLLSW}, as well as the generalised Serre weight conjectures in \cite{DS}. Theorem~\ref{cor-irred} suggests that, whatever shape this representation theoretic input takes, its simplest building blocks should occur with weights bounded of order $p^{N_\tau+1}$, for $N_\tau$ the measure of irregularity as in the theorem.
	
	We conclude by briefly discussing the methods used to establish Theorem~\ref{thm-1.1}. In spirit they are identical to those employed in \cite{B21,B23} and aim to control moduli spaces of crystalline Breuil--Kisin modules of weight $\mu$ in terms of the relative position of the image of Frobenius. At least when $K = K_0$, this leads to a description of these moduli spaces in terms of a flag variety, which is the essential source of irreducibility. To understand how these techniques apply in weights beyond $[-p,0]$ consider the prototypical example where $[K:\mathbb{Q}_p] =n$ and the Hodge--Tate weights are zero at all embeddings besides one. Then the matrix of Frobenius on any Breuil--Kisin module with these weights, which is usually semilinear for an endomorphism lifting the $p$-th power map mod $p$, can be interpreted as semilinear for an endomorphism lifting the $p^n$-th power map. These additional powers of $p$ mean the various ``extension by Frobenius'' techniques from \cite{B21,B23} become powerful enough to overcome the bound $[-p,0]$.
	
	Essentially all the calculations in the paper are, just as in \cite{B23}, done for ramified extensions. The requirement that $K$ be unramified only arises when analysing certain subschemes inside the affine grassmannian; the problem being that, because here there is no Frobenius, irregularity assumptions no longer give any additional leverage. However, in the unramified case these subschemes are much simpler and identify flag varieties so there is no issue. Analogous results in the ramified setting would require an extension of \cite[9.1]{B23} which would certainly need a new approach (at least when $n>2$).  We also note that, with alterations made as in \cite{B23}, analogous statements to Theorem~\ref{thm-1.1} will also hold for moduli of crystalline representation valued in any split reductive group.

\section{Notation}\label{sec-notation}
 Let $K / \mathbb{Q}_p$ be a finite extension with residue field $k$ and completed algebraic closure $C$. Let $K_0 \subset K$ denote the maximal unramified sub-extension. Fix $\pi \in K$ a choice of uniformiser and set $E(u) = u^e +\ldots \in W(k)[u]$ the minimal polynomial of $\pi$ over $W(k)$. Thus $e$ denotes the ramification degree of $K$ over $\mathbb{Q}_p$. We also fix compatible systems $\pi^{1/p^\infty}$ and $\epsilon_\infty$ in $C$ with $\pi^{1/p^\infty}$ a sequence of $p$-th power roots of $\pi$ and $\epsilon_\infty$ a system of primitive $p$-th power roots of $1$.

Fix $\mathcal{O}$ the ring of integers in a finite extension of $\mathbb{Q}_p$, and assume this extension is sufficiently large that it contains all Galois conjugates of $K$. This ensures any $K \otimes_{\mathbb{Z}_p} \mathcal{O}$-module $M$ decomposes as a product $\prod_{\kappa:K \hookrightarrow \mathcal{O}[\frac{1}{p}]} M_\kappa$ according to the action of $K$ on $M$. Similarly, any $W(k) \otimes_{\mathbb{Z}_p} \mathcal{O}$ module $M$ decomposes as a product $\prod_{\tau:K_0 \hookrightarrow \mathcal{O}[\frac{1}{p}]} M_\tau$ (this crucially uses that $W(k)$ is unramified over $\mathbb{Z}_p$ and would not hold for $\mathcal{O}_K$ unless $K =K_0$). Set $\mathbb{F}$ equal the residue field of $\mathcal{O}$ and $\mathfrak{m}_{\mathcal{O}}$ its maximal ideal.

\section{Crystalline Breuil--Kisin modules}\label{sec-crysBK}

\begin{definition}
For any $p$-adically complete $\mathcal{O}$-algebra $A$ we write $Z_d(A)$ for the groupoid of rank $d$ Breuil--Kisin modules over $A$, i.e. the groupoid of projective $\mathfrak{S}_A := (W(k) \otimes_{\mathbb{Z}_p} A)[[u]]$-modules of rank $d$ equipped with an isomorphism
$$
\varphi:\varphi^*\mathfrak{M}[\tfrac{1}{E(u)}] \cong \mathfrak{M}[\tfrac{1}{E(u)}]
$$
Here $\varphi$ denotes the $A$-linear endomorphism which lifts the $p$-th power map on $W(k)$ and sends $u \mapsto u^p$, and we frequently view the above isomorphism as a $\varphi$-semilinear map $\mathfrak{M} \rightarrow \mathfrak{M}[\frac{1}{E(u)}]$ via $m \mapsto \varphi(m \otimes1)$. Any homomorphism $A \rightarrow B$ induces a functor $Z_d(B) \rightarrow Z_d(A)$ by pullback along $\mathfrak{S}_A \rightarrow \mathfrak{S}_B$, and this makes $Z_d$ into a category fibred over $\operatorname{Spf}\mathcal{O}$.
\end{definition}

When $A$ is additionally a $p$-adically complete $\mathcal{O}$-algebra topologically of finite type then we also consider the $\mathfrak{S}_A$-algebra $A_{\operatorname{inf},A}$ defined as in e.g. \cite[2.2]{EG19} or \cite[14.1]{B23}. This is equipped with an $A$-linear Frobenius automorphism extending that on $\mathfrak{S}_A$ and an $A$-linear continuous action of $G_K$ commuting with the Frobenius. 

\begin{definition}
A crystalline $G_K$-action on $\mathfrak{M}$ is a continuous $A_{\operatorname{inf},A}$-semilinear $\varphi$-equivariant action of $G_K$ on $\mathfrak{M} \otimes_{\mathfrak{S}_A} A_{\operatorname{inf},A}$ satisfying
	\begin{equation}\label{eq-Galoisaction}
		\sigma(m)-m \in \mathfrak{M} \otimes_{\mathfrak{S}_{A}} u \varphi^{-1}(\mu) A_{\operatorname{inf},A}, \qquad \sigma_\infty(m) - m = 0
	\end{equation}
	for all $m \in \mathfrak{M}$ and all $\sigma  \in G_K,\sigma_\infty \in G_{K_\infty}$. Write $Y_d(A)$ for the groupoid consisting of $\mathfrak{M} \in Z_d(A)$ equipped with a crystalline $G_K$-action, and $Y_d$ for the resulting limit preserving category fibred over $\operatorname{Spf}\mathcal{O}$.
\end{definition}

	If $A$ is finite and flat over $\mathcal{O}$ then any $\mathfrak{M} \in Y_d(A)$ gives rise to a crystalline representation $T$ on a finite free $A$-module of rank $d$ \cite[16.2]{B23}. Associated to $T$ is the filtered $\varphi$-module $D =(T\otimes_{\mathbb{Z}_p} B_{\operatorname{crys}})^{G_K}$. This is a $K_0 \otimes_{\mathbb{Z}_p} A$-module of rank $d$ equipped with a $\varphi$-semilinear bijection and a separated, exhaustive, and decreasing filtration\footnote{i.e. a filtration $\ldots \subset \operatorname{Fil}^{i+1} \subset \operatorname{Fil}^i \subset \operatorname{Fil}^{i-1} \subset \ldots$ with $\operatorname{Fil}^N =0$ for $N>>0$ and $\operatorname{Fil}^M = D_K$ for $M <<0$. } on $D_K = D \otimes_{K_0} K$ by $K \otimes_{\mathbb{Z}_p} A$-submodules with $A$-projective graded pieces \cite[1.3.4]{KisFF}. This filtration arises from an identification $D_K = (T \otimes_{\mathbb{Z}_p} B_{\operatorname{dR}})^{G_K}$ and is defined by
$$
\operatorname{Fil}^i(D_K) = (T \otimes t^iB_{\operatorname{dR}}^+)^{G_K}
$$ 
for $t \in B_{\operatorname{dR}}^+$ a generator of its maximal ideal and we say $T$ has Hodge type $\mu$ if the  filtration on $D_K = D \otimes_{K_0} K$ has associated grading of type $\mu$.

\begin{example}
	If $T$ is the one dimensional representation on which the cyclotomic character acts then $0 = \operatorname{Fil}^0(D_K) \subset \operatorname{Fil}^{-1}(D_K) = D_K$ and so the grading on the Hodge type of $T$ is concentrated in degrees $[-1,0]$.
\end{example}
\begin{proposition}
	For each Hodge type $\mu$ there exists a closed subfunctor $Y^\mu_d$ of $Y_d$ which is represented by an $\mathcal{O}$-flat $p$-adic algebraic formal stack (in the sense \cite[A7]{EG19}) of topologically finite type over $\mathcal{O}$. It is uniquely determined by the property that its groupoid of $A$-valued points, for any finite flat $\mathcal{O}$-algebra $A$, is canonically equivalent to the full subcategory $Y_d^\mu(A)$ consisting of those $\mathfrak{M} \in Y_d(A)$ whose associated crystalline representation has Hodge type $\mu$.
\end{proposition}
\begin{proof}
See \cite[16.3]{B23} or \cite[4.8.2]{EG19}.	
\end{proof}

\section{Affine grassmannian}\label{sec-affinegrass}
A useful means of understanding $Z_d$ and $Y_d$ is via the relative position of $\mathfrak{M}$ and $\varphi^*\mathfrak{M}$. With this in mind we introduce the ind-scheme $\operatorname{Gr}$ over $\operatorname{Spec}\mathcal{O}$ whose $A$-valued points classify rank $d$ projective $(W(k) \otimes_{\mathbb{Z}_p} A)[u]$-modules  together with a choice of basis after inverting $E(u)$ (equivalently rank $d$ projective $(W(k) \otimes_{\mathbb{Z}_p} A)[u]$-submodules of $(W(k) \otimes_{\mathbb{Z}_p} A)[u,\frac{1}{E(u)}]^d$). Recall the Beauville--Laszlo glueing lemma \cite{BL} which asserts that, in the above description, the ring $(W(k) \otimes_{\mathbb{Z}_p} A)[u]$ can be replaced by its $E(u)$-adic completion. In particular, if $A$ is $p$-adically complete then
\begin{itemize}
	\item the $A$-valued points of $\operatorname{Gr}$ can be viewed as projective modules over $\mathfrak{S}_A$,
	\item the $A[\frac{1}{p}]$-valued points can be viewed as projective modules over the $E(u)$-adic completion of $(K_0 \otimes_{\mathbb{Z}_p} A)[u]$, which we denote $\widehat{\mathfrak{S}}_{A[\frac{1}{p}]}$.
\end{itemize}
The first bullet point provides the link between $Z_d$ and $\operatorname{Gr}$. There is a diagram
$$
\begin{tikzcd}
	& \ar[dl,"\Gamma"] \widetilde{Z}_d \ar[dr,"\Psi"] & \\
	Z_d & & \operatorname{Gr}
\end{tikzcd}
$$ 
where we write $\operatorname{Gr}$ also for its $p$-adic completion and $\widetilde{Z}_d$ denotes the category fibred over $\operatorname{Spf}\mathcal{O}$ whose $A$-points classify pairs $(\mathfrak{M},\iota)$ with $\mathfrak{M} \in Z_d(A)$ and $\iota$ a choice of $\mathfrak{S}_A$-basis. The map $\Gamma$ forgets the choice of basis and $\Psi(\mathfrak{M},\iota) = \mathfrak{M}$ together with the basis $\varphi(\iota)$ of $\mathfrak{M}[\frac{1}{E(u)}]$. Observe that $\Gamma$ and $\Psi$ are respectively torsors for two actions of the formal group scheme $L^+\operatorname{GL}_d$ over $\operatorname{Spf}\mathcal{O}$ with $A$-points $\operatorname{GL}_d(\mathfrak{S}_A)$
\begin{equation}\label{eq-twoactions}
g \cdot_{\varphi} (\mathfrak{M},\iota) = (\mathfrak{M},\iota g), \qquad g \cdot_{\operatorname{trans}} (\mathfrak{M},\iota) = (\mathfrak{M}_g,\iota)
\end{equation}
where $\mathfrak{M}_g = \mathfrak{M}$ as $\mathfrak{S}_A$-modules but with the twisted Frobenius $\varphi_g(\iota) = \iota g C$ for $C$ defined by $\varphi(\iota) = \iota C$. If $L\operatorname{GL}_d$ denotes the functor over $\operatorname{Spf}\mathcal{O}$ given by $A \mapsto \operatorname{GL}_d(\mathfrak{S}_A[\frac{1}{E(u)}])$ then $(\mathfrak{M},\iota) \mapsto C$ for $\varphi(\iota) =\iota C$ identifies $\widetilde{Z}_d \cong L\operatorname{GL}_d$. Under this identification $\cdot_\varphi$ corresponds to the action $g \cdot_\varphi C = g C\varphi(g^{-1})$ and $\cdot_{\operatorname{trans}}$ to the action $g \cdot_{\operatorname{trans}} C = g C$.

\begin{lemma}\label{lem-filt}
	Let $\mu$ be a Hodge type and $\operatorname{FL}_\mu$ the projective $\mathcal{O}[\frac{1}{p}]$-scheme whose $A$-points classify separated, exhaustive, and descending filtrations on $(K\otimes_{\mathbb{Q}_p} A)^d$ with associated graded of type $\mu$. Then
	$$
	\operatorname{Fil}^\bullet \mapsto \sum_{i \in \mathbb{Z}} \operatorname{Fil}^i \otimes_{K \otimes_{\mathbb{Q}_p} A} E(u)^{-i} \widehat{\mathfrak{S}}_A \subset \widehat{\mathfrak{S}}_A[\tfrac{1}{E(u)}]^d
	$$
	defines a closed immersion $\operatorname{FL}_\mu \rightarrow \operatorname{Gr}[\frac{1}{p}]$.
\end{lemma}
\begin{proof}
	Clearly the morphism is a monomorphism and, since $\operatorname{FL}_\mu$ is $\mathcal{O}[\frac{1}{p}]$-projective and $\operatorname{Gr}$ is $\mathcal{O}$-separated, it is also proper. Proper monomorphisms are closed immersions so the lemma follows.
\end{proof}

\begin{definition}
Let $M_\mu \subset \operatorname{Gr}$ be the scheme theoretic image of the composite $\operatorname{FL}_\mu \rightarrow \operatorname{Gr}[\frac{1}{p}] \rightarrow \operatorname{Gr}$, i.e. the closure of $\operatorname{FL}_\mu$ in $\operatorname{Gr}$.
\end{definition}

In general, the geometry of the $M_\mu$'s can be quite complicated, see for example \cite[1.1]{B23}. However, when $K = K_0$ this is not the case and they can be described easily:

\begin{lemma}\label{lem-MmuK=K0}
	Suppose $K = K_0$. Then the Hodge type $\mu$ descends to a unique isomorphism class of graded $W(k) \otimes_{\mathbb{Z}_p} \mathcal{O}$-modules and there is an isomorphism
	$$
	\operatorname{FL}_{\mu,\mathcal{O}} \xrightarrow{\sim} M_\mu, \qquad\operatorname{Fil}^\bullet \mapsto \sum_{i \in \mathbb{Z}} \operatorname{Fil}^i \otimes_{W(k)\otimes_{\mathbb{Z}_p} A} E(u)^{-i} \mathfrak{S}_A
	$$
	where $\operatorname{FL}_{\mu,\mathcal{O}}$ is the smooth irreducible projective $\mathcal{O}$-scheme whose $A$-points classify separated exhaustive decreasing filtrations on $(W(k) \otimes_{\mathbb{Z}_p} A)^d$ with associated graded of type $\mu$.
\end{lemma}  
\begin{proof}
	That $\mu$ descends to an integral isomorphism class of graded modules follows from the fact the decomposition $K \otimes_{\mathbb{Z}_p} A \cong \prod_{\kappa:K \rightarrow \mathcal{O}[\frac{1}{p}]} A[\frac{1}{p}]$ descends to an integral decomposition  $\mathcal{O}_K \otimes_{\mathbb{Z}_p} A \cong \prod_{\kappa:K \rightarrow \mathcal{O}[\frac{1}{p}]} A$ when $K = K_0$. By the same argument as in Lemma~\ref{lem-filt} the given morphism $\operatorname{FL}_{\mu,\mathcal{O}} \rightarrow \operatorname{Gr}$ is a closed immersion. By flatness, and the fact it identfies $\operatorname{FL}_{\mu,\mathcal{O}}$ and $M_\mu$ after inverting $p$, we deduce the lemma.
\end{proof}
\section{The main results}

The following is then the our main theorem.
\begin{theorem}\label{thm-main}
	Assume $K(\pi^{1/p^\infty}) \cap K(\epsilon_\infty) =K$ and suppose $\mu$ is a Hodge type for which there exist integers $N_\tau \geq 0$ for $\tau:K_0 \hookrightarrow\mathcal{O}[\frac{1}{p}]$ such that the grading on $\mu_\kappa$, the $\kappa$-th part of $\mu$,
	\begin{itemize}
		\item  is concentrated in degree $0$ whenever $\kappa|_{K_0} = \tau \circ \varphi^n$ with $1 \leq n \leq N_\tau$
		\item	is concentrated in degree $[-a_\kappa,0]$ whenever $\kappa|_{K_0} = \tau$ for $a_\kappa$ satisfying
		$$
		\sum_{\kappa|_{K_0} = \tau} a_\kappa \leq \frac{p^{N_\tau+1}-1}{\nu}+1,\qquad \nu = \operatorname{max} v_\pi(\pi-\pi')
		$$
		where $v_\pi$ denotes the $\pi$-adic valuation and the maximum runs over distinct conjugates $\pi'$ of $\pi$ in $\mathcal{O}_K$. If $N_\tau =0$ for each $\tau$ and $\nu =1$ then assume also that this inequality is strict for at least one $\tau$.
	\end{itemize}
Then:
	\begin{enumerate}
		\item There exists a closed algebraic substack $Z^\mu_{d,\mathbb{F}} \subset Z_d \otimes_{\mathcal{O}} \mathbb{F}$ of dimension $\operatorname{dim}\operatorname{FL}_\mu$ defined by the property that any $\mathfrak{S}_A$-free $\mathfrak{M} \in Z_d(A)$ is contained in $Z^\mu_{d,\mathbb{F}}$ if and only if there exists an $\mathfrak{S}_A$-basis $\iota$ of $\mathfrak{M}$ with 
		$$
		\Psi(\mathfrak{M},\iota) \in M_\mu(A)
		$$
		\item The morphism $Y^\mu_d \otimes_{\mathcal{O}} \mathbb{F} \rightarrow Z_d$ is a closed immersion factoring through $Z^\mu_{d,\mathbb{F}}$.
		\item If $K = K_0$ then $Z^\mu_{d,\mathbb{F}}$ is irreducible and $\mathcal{O}$-smooth, and $Y^\mu_d \otimes_{\mathcal{O}} \mathbb{F} \hookrightarrow Z^\mu_{d,\mathbb{F}}$ is an isomorphism.
	\end{enumerate}
\end{theorem}

The assumption $K(\pi^{1/p^\infty}) \cap K(\epsilon_{\infty})=K$ is essentially harmless. If $p>2$ it is automatic, and when $p=2$, while not automatic, $\pi$ can always be chosen so that it holds, see \cite[2.1]{Wang17}. Notice also that if the ramification degree $e$ of $K$ of $\mathbb{Q}_p$ is prime to $p$ (i.e. if $K$ is tamely ramified) then $\nu =1$. In this case the bound in the theorem becomes $\sum_{\kappa|_{K_0} = \tau} a_\kappa \leq p^{N_\tau +1}$.

\begin{proof}
	The proof will take the rest of the paper. The construction and properties of $Z^\mu_{d,\mathbb{F}}$ are established in Proposition~\ref{prop-stability}. That $Y^\mu_d \otimes_{\mathcal{O}} \mathbb{F} \rightarrow Z_d$ factors through $Z^\mu_{d,\mathbb{F}}$ is then a consequence of Corollary~\ref{cor-factor}. The fact that this is a closed immersion is proved in Corollary~\ref{cor-closedimm} and the fact it is an isomorphism when $K = K_0$ follows from Proposition~\ref{thm-isom}.
\end{proof}

Since the crystalline locus $\mathcal{X}^\mu_d$ from the introduction is defined in \cite[4.8.8]{EG19} as the scheme theoretic image of $Y^\mu_d$ under a morphism $Y_d \rightarrow \mathcal{X}_d$, the main result of the paper follows easily:

\begin{corollary}\label{cor-irred}
		If $K = K_0$ and $\mu$ is as in Theorem~\ref{thm-main} then $\mathcal{X}^\mu_d \otimes_{\mathcal{O}} \mathbb{F}$ is irreducible.
\end{corollary}
\begin{proof}
	The formation of scheme theoretic images is not compatible with non-flat base change so the scheme theoretic image of $Y^\mu_d \otimes_{\mathcal{O}} \mathbb{F}$ is only a closed substack of $\mathcal{X}^\mu_d \otimes_{\mathcal{O}} \mathbb{F}$. However, these closed substacks have the same topological space (in the sense of \cite[04XE]{stacks-project}) because every finite type point of $\mathcal{X}^\mu_d \otimes_{\mathcal{O}} \mathbb{F}$ (in the sense of \cite[06FW]{stacks-project}) arises from a finite type of $Y^\mu_d$ which necessarily factors through $Y^\mu_d \otimes_{\mathcal{O}} \mathbb{F}$. Since $Y^\mu_d \otimes_{\mathcal{O}} \mathbb{F}$ is irreducible the same is true of its scheme theoretic image, and hence also of $\mathcal{X}^\mu_d \otimes_{\mathcal{O}} \mathbb{F}$.
\end{proof}

\section{Filtrations and the shape of Frobenius}\label{sec-shape}

	Let $A$ be a finite flat $\mathcal{O}$-algebra $A$, $\mathfrak{M} \in Y^\mu_d(A)$, and $D$ the associated filtered $\varphi$-module, as in Section~\ref{sec-crysBK}.

\begin{proposition}\label{prop-kis}
	\begin{itemize}
		\item There exists a $\varphi$-equivariant isomorphism
		$$
		\xi: \varphi^* \mathfrak{M} \otimes_{\mathfrak{S}_A} \mathcal{O}^{\operatorname{rig}}_A[\tfrac{1}{\varphi(\lambda)}] \cong D \otimes_{K_0 \otimes_{\mathbb{Z}_p} A} \mathcal{O}^{\operatorname{rig}}_A[\tfrac{1}{\varphi(\lambda)}]
		$$
		where $\mathcal{O}^{\operatorname{rig}} \subset K_0[[u]]$ is the subring of power series converging on the open unit disk, $\mathcal{O}^{\operatorname{rig}}_A = \mathcal{O}^{\operatorname{rig}} \otimes_{\mathbb{Z}_p} A$, and $\lambda = \prod_{n \geq 0} \varphi^n(\frac{E(u)}{E(0)})$.
		\item After basechanging to $\widehat{\mathfrak{S}}_A[\frac{1}{E(u)}]$, $\xi$ induces an identification
		$$
		\mathfrak{M} \otimes_{\mathfrak{S}} \widehat{\mathfrak{S}}_{A[\frac{1}{p}]} = \sum_{i\in \mathbb{Z}} \operatorname{Fil}^i(D \otimes_{K_0} K) \otimes_{K \otimes_{\mathbb{Z}_p} A} E(u)^{-i}\widehat{\mathfrak{S}}_{A[\frac{1}{p}]}
		$$
		(here we use the  unique action of $K$ on $\widehat{\mathfrak{S}}_{A[\frac{1}{p}]}$ lifting that on its residue field).
	\end{itemize}
\end{proposition}
\begin{proof}
	See \cite[1.2.1]{Kis06} and \cite[1.2.6]{Kis06}.
\end{proof}

The following is then an immediate consequence of Proposition~\ref{prop-kis}.

\begin{corollary}\label{cor-XMmu}
	Assume that $\mathfrak{M}$ admits an $\mathfrak{S}_A$-basis $\iota$ and let $\beta$ be a $K_0 \otimes_{\mathbb{Z}_p} A$-basis of $D$. If $X \in \operatorname{GL}_d(\mathcal{O}^{\operatorname{rig}}_A[\frac{1}{\varphi(\lambda)}])$ is such that $\xi(\varphi(\iota)) = \beta X$ then
	$$
	X \cdot \Psi(\mathfrak{M},\iota)[\tfrac{1}{p}] \in M_\mu(A[\tfrac{1}{p}])
	$$
	Here the action of $X$ on $\operatorname{Gr}(A[\frac{1}{p}])$ is via the interpretation of $A[\frac{1}{p}]$-valued points as $\widehat{\mathfrak{S}}_{A[\frac{1}{p}]}$-submodules of $\widehat{\mathfrak{S}}_{A[\frac{1}{p}]}[\frac{1}{E(u)}]^d$.
\end{corollary}

\section{Adapted bases}

If $\mathfrak{M} \in Z_d(A)$ and $\mathfrak{M}_\tau$ denotes the $A[[u]]$-submodule upon which $W(k)$ acts through $\tau$ then the Frobenius on $\mathfrak{M}$ restricts to a semilinear map $\mathfrak{M}_{\tau \circ \varphi} \rightarrow \mathfrak{M}_\tau[\frac{1}{\tau(E(u))}]$.

\begin{definition}
If $N_\tau \geq 0$ are given for each $\tau:K_0 \hookrightarrow \mathcal{O}[\frac{1}{p}]$ and $\mathfrak{M} \in Z_d(A)$ then we say an $\mathfrak{S}_A$-basis $\iota$ of $\mathfrak{M}$ is $(N_\tau)$-adapted if $\iota$ arises from $A[[u]]$-bases $\iota_\tau$ of $\mathfrak{M}_\tau$ satisfying 
$$
\varphi(\iota_{\tau \circ \varphi^{n+1}}) = \iota_{\tau \circ \varphi^n}
$$
for $1\leq n\leq N_\tau$. In particular, $\iota_{\tau \circ \varphi} = \varphi^{N_\tau}(\iota_{\tau \circ \varphi^{N_\tau+1}})$.
\end{definition}

\begin{lemma}\label{lem-adaptedbasis}
	Let $\mathfrak{M} \in Z_d(A)$ be $\mathfrak{S}_A$-free and suppose $\mu$ is a Hodge type for which there exist integers $N_\tau \geq 0$ with
	\begin{itemize}
		\item the grading on $\mu_\kappa$ concentrated in degree $0$ whenever $\kappa|_k = \tau \circ \varphi^n$ for $1\leq n \leq N_\tau$.
	\end{itemize} 
Then $\mathfrak{M}$ admits an $(N_\tau)$-adapted basis in the following two situations:
	\begin{enumerate}
		\item $\Psi(\mathfrak{M},\iota) \in M_\mu(A)$ for some $\mathfrak{S}_A$-basis $\iota$ and $A$ any $p$-adically complete $\mathcal{O}$-algebra.
		\item $A$ is a finite flat $\mathcal{O}$-algebra and $\mathfrak{M} \in Y^\mu_d(A)$.
	\end{enumerate}
\end{lemma}
\begin{proof}
	Since $\mathfrak{M}$ is $\mathfrak{S}_A$-free we can always choose a basis $\iota$ arising from $A[[u]]$-bases $\iota_\tau$ of $\mathfrak{M}_\tau$. If $\varphi(\iota_{\tau \circ \varphi}) = \iota_\tau C_\tau$ for $C_\tau \in \operatorname{GL}_d(A[[u]][\frac{1}{\tau(E(u))}])$ then the existence of an $(N_\tau)$-adapted basis is ensured by the fact that $C_{\tau \circ \varphi^n} \in \operatorname{GL}_d(A[[u]])$ for $1 \leq n\leq N_\tau$.
	
	If $\mathfrak{M} \in Y^\mu_d(A)$ has associated filtered $\varphi$-module $D$ then the grading on $D_{K,\tau}$ is concentrated in degree zero if and only if the $\tau$-th part of $\sum_{i \in \mathbb{Z}} \operatorname{Fil}^i(D_K) \otimes_{K\otimes_{\mathbb{Z}_p} A} E(u)^{-i}\widehat{\mathfrak{S}}_{A[\frac{1}{p}]}$ and the $\tau$-th part of $D_K \otimes_{K \otimes_{\mathbb{Z}_p} A} \widehat{\mathfrak{S}}_{A[\frac{1}{p}]}$ coincide. By Proposition~\ref{prop-kis} this occurs if and only if the $\tau$-th parts of $\varphi^*\mathfrak{M} \otimes_{\mathfrak{S}_A} \widehat{\mathfrak{S}}_{A[\frac{1}{p}]}$ and $\mathfrak{M} \otimes_{\mathfrak{S}_A} \widehat{\mathfrak{S}}_{A[\frac{1}{p}]}$ coincide. This occurs if and only if $C_\tau$ is contained in the $\tau$-th part of $\operatorname{GL}_d(\widehat{\mathfrak{S}}_{A[\frac{1}{p}]})$, which is equivalent to $C_\tau \in \operatorname{GL}_d(A[[u]])$. When $\Psi(\mathfrak{M},\iota) \in M_\mu(A)$ the argument is similar.
\end{proof}

\section{Integrality}\label{sec-int}

Here we produce integrality constraints on the isomorphism $\xi$ from Proposition~\ref{prop-kis}. These generalise \cite[\S 17]{B23}. To do so observe that, since the $E(u)$-adic completion $\widehat{\mathfrak{S}}_{A[\frac{1}{p}]}$ of $(K_0 \otimes_{\mathbb{Z}_p} A)[u]$ from Section~\ref{sec-affinegrass} is a $K \otimes_{\mathbb{Z}_p} A$-algebra, we have a decomposition  $\widehat{\mathfrak{S}}_{A[\frac{1}{p}]} \cong \prod_{\kappa:K \rightarrow \mathcal{O}[\frac{1}{p}]} \widehat{\mathfrak{S}}_{A[\frac{1}{p}],\kappa}$ according to the action of $K$. We can identify  $\widehat{\mathfrak{S}}_{A[\frac{1}{p}],\kappa}$ with the $u -\kappa(\pi)$-th completion of $A[\frac{1}{p},u]$, which in turn identifies with the power series ring over $A[\frac{1}{p}]$ in the variable $u-\kappa(\pi)$. We then refer to the image of $f \in \mathcal{O}^{\operatorname{rig}}[\frac{1}{\varphi(\lambda)}]$ under the composite
$$
\mathcal{O}^{\operatorname{rig}}_A[\tfrac{1}{\varphi(\lambda)}] \rightarrow \widehat{\mathfrak{S}}_{A[\frac{1}{p}]} \rightarrow \widehat{\mathfrak{S}}_{A[\frac{1}{p}],\kappa} = A[\tfrac{1}{p}][[u-\kappa(\pi)]]
$$
as its Taylor expansion around $u =\kappa(\pi)$. One has the following easy lemma:

\begin{lemma}\label{lem-truncate}
	Suppose $X \in \operatorname{GL}_d(\widehat{\mathfrak{S}}_{A[\frac{1}{p}]})$ has Taylor expansion around $u = \kappa(\pi)$ congruent to $1$ modulo $(u-\kappa(\pi))^{a_\kappa}$ for $a_\kappa \geq 0$. Then $X$ acts trivially on $M_\mu[\frac{1}{p}]$ whenever $\mu$ is a Hodge type with $\kappa$-th part concentrated in degrees $[-a_\kappa,0]$.
\end{lemma}
\begin{proof}
	The $\kappa$-th-part of $\sum_{i \in\mathbb{Z}} \operatorname{Fil}^i \otimes_{K\otimes_{\mathbb{Z}_p} A} E(u)^{-i} \widehat{\mathfrak{S}}_A$ is $\sum_{i \in\mathbb{Z}} \operatorname{Fil}^i_\kappa \otimes_{A} (u-\kappa(\pi))^{-i} \widehat{\mathfrak{S}}_{A[\frac{1}{p}],\kappa}$ and we have to show that $X-1$ maps $(u-\kappa(\pi))^{-j}\operatorname{Fil}^j_\kappa$ into $\sum_{i \in\mathbb{Z}} \operatorname{Fil}^i_\kappa \otimes_{A} (u-\kappa(\pi))^{-i} \widehat{\mathfrak{S}}_{A[\frac{1}{p}],\kappa}$ for all $j$ when $\operatorname{Fil}^\bullet$ has type $\mu$. The assumption on $\mu$ implies $\operatorname{Fil}_\kappa^j =0$ for $j > 0$ and $\operatorname{Fil}^j_\kappa = A[\frac{1}{p}]^d$ for $j \leq -a_\kappa$. The assertion is therefore immediate when $j > 0$ and, when $j \leq 0$, the assumption on the Taylor expansion of $X$ implies $X-1$ maps $(u-\kappa(\pi))^{-j}\operatorname{Fil}^j_\kappa$ into $(u-\kappa(\pi))^{-j+a_\kappa} \widehat{\mathfrak{S}}_{A[\frac{1}{p}],\kappa}^d =\operatorname{Fil}^{-a_\kappa+j} \otimes_{A[\frac{1}{p}]} (u-\kappa(\pi))^{-j+a_\kappa} \widehat{\mathfrak{S}}_{A[\frac{1}{p}],\kappa}$ as required.
\end{proof}
Though $X$ in Corollary~\ref{cor-XMmu} will never be $p$-adically integral we will be able to impose some integrality on the first few terms of its Taylor expansions. This, combined with the previous lemma, allows us to replace $X$ in Corollary~\ref{cor-XMmu} with an element in $\operatorname{GL}_d(\mathfrak{S}_A)$.
\begin{proposition}\label{prop-intxi} 
	Assume $K(\pi^{1/p^\infty}) \cap K(\epsilon_\infty) = K$ and $\mathfrak{M} \in Y^\mu(A)$ for $A$ a finite flat $\mathcal{O}$-algebra. For each $\tau:K_0 \hookrightarrow \mathcal{O}[\frac{1}{p}]$ suppose there are $N_\tau \geq 0$ such that
	\begin{itemize}
		\item the grading on $D_{K,\tau\circ \varphi^n}$ is concentrated in degree $0$ for $1 \leq n \leq N_\tau$
	\end{itemize} 
	If $\iota$ is an $(N_\tau)$-adapted basis and $X \in \operatorname{GL}_d(\mathcal{O}^{\operatorname{rig}}_A[\tfrac{1}{\varphi(\lambda)}])$ is such that $\xi(\varphi(\iota)) =\beta X$ then the Taylor expansion 
	$$
	\sum_{i \geq 0} (u-\kappa(\pi))^i X_i
	$$
	of $X$ around $u = \kappa(\pi)$ has $X_0^{-1} X_i \in \frac{p^{N_\tau}}{i}\kappa(\pi)^{p^{N_\tau+1} - i} \operatorname{Mat}(A)$ for $0 < i \leq p^{N_{\tau+1}}$.
\end{proposition}

\begin{proof}
	We begin by considering a differential operator $N_\nabla$ on $D \otimes_{K_0} \mathcal{O}^{\operatorname{rig}}[\frac{1}{\varphi(\lambda)}]$ defined by 
	$$
	N_{\nabla}(d \otimes f) = d \otimes \partial(f), \qquad \partial = u\frac{d}{du}
	$$
	for $d \in D$ and $f \in \mathcal{O}^{\operatorname{rig}}[\frac{1}{\varphi(\lambda)}]$.	Notice that $N_\nabla$ is $\varphi$-equivariant, $A$-linear, and satisfies $p \varphi \circ N_\nabla = N_\nabla \circ \varphi$  (because $p \varphi \circ \partial = \partial \circ \varphi$). Under the isomorphism $\xi$ the operator $N_\nabla$ corresponds to a differential operator on $\varphi^*\mathfrak{M} \otimes_{\mathfrak{S}} \mathcal{O}^{\operatorname{rig}}[\frac{1}{\varphi(\lambda)}]$ which we again denote by $N_\nabla$. If $N_\nabla(\varphi(\iota)) = \varphi (\iota) N$ for $N \in \operatorname{Mat}(\mathcal{O}^{\operatorname{rig}}_A[\frac{1}{\varphi(\lambda)}])$ then the identify $N_\nabla \circ \xi = \xi \circ N_\nabla$ implies
$X N = \partial(X)$.
	In terms of Taylor expansions around $u= \kappa(\pi)$, if $N = \sum_{i \geq 0} (u-\kappa(\pi))^i N_i$, this gives the recurrence 
	$$
 \sum_{i+j = n} X_i N_j = 	n X_n + \kappa(\pi)(n+1) X_{n+1} 
	$$
	or equivalently, $n X_0^{-1} X_n + \kappa(\pi)(n+1)X_0^{-1} X_{n+1} = \sum_{i+j = n} (X_0^{-1}X_i) N_j$. Now suppose that
$$
N_i \in p^{N_{\tau}}\kappa(\pi)^{p^{N_\tau+1} - i}\operatorname{Mat}(A)
$$
for $0 \leq i \leq p^{N_{\tau+1}}$ and that $X_0^{-1} X_i \in \frac{p^{N_{\tau}}\kappa(\pi)^{p^{N_\tau+1} - i}}{i} \operatorname{Mat}(A)$ for $0<i \leq n <p^{N_\tau+1}$. Then the entries of $\sum_{i+j = n} (X_0^{-1}X_i) N_j$ are contained in
$$
 \sum_{0\leq  j \leq n-1} p^{2N_\tau}\frac{\kappa(\pi)^{2p^{N_{\tau}+1}-n}}{n-j}\operatorname{Mat}(A) +  \underbrace{p^{N_{\tau}}\kappa(\pi)^{p^{N_\tau+1}-n}\operatorname{Mat}(A)}_{j=n}
$$ 
This is contained in $p^{N_\tau} \kappa(\pi)^{p^{N_\tau+1}- n} \operatorname{Mat}(A)$ because $p^{N_\tau}\frac{\kappa(\pi)^{p^{N_\tau+1}}}{n-j}$ is integral whenever $0 < n-j < p^{N_\tau+1}$. Inducting on $n$ therefore shows that the above divisibility of the $N_i$'s implies the proposition.

An advantage of $N_\nabla$ over $\xi$ is that, while its construction requires $p$ to be inverted, it can also be described in terms of the (integral) Galois action on the lattice $T$ (see \cite[20.9]{B23}). In this way one obtains the following constraints on $N_\nabla$:
	
	\begin{proposition}\label{prop-Nablaintegral}
		If $m_0 \in \mathfrak{M}$. Then 
		$$
		N_\nabla(m_0) \in \mathfrak{M} \otimes_{\mathfrak{S}} \tfrac{u}{p} S_{\operatorname{max}}
		$$
		where $S_{\operatorname{max}} := \mathcal{O}^{\operatorname{rig}}[\frac{1}{\varphi(\lambda)}] \cap W(k)[u,\frac{u^e}{p}]$.
	\end{proposition}
\begin{proof}
	See \cite[20.1]{B23}. This is where we use that $K(\pi^{1/p^\infty}) \cap K(\epsilon_{\infty}) = K$.
\end{proof}

Iterating the identity $p  \varphi \circ N_\nabla = N_\nabla \circ \varphi$ gives $N_\nabla \circ \varphi^{N_\tau+1} = p^{N_\tau+1} \varphi^{N_\tau+1} \circ N_\nabla$. If $\iota$ is an $(N_\tau)$-adapted basis of $\mathfrak{M}$ then $\varphi(\iota_{\tau \circ \varphi}) = \varphi^{N_\tau+1}(\iota_{\tau \circ \varphi^{N_\tau+1}})$ and so Proposition~\ref{prop-Nablaintegral} implies $N_\nabla(\varphi(\iota_\tau \circ \varphi)) \subset \mathfrak{M} \otimes_{\mathfrak{S}} p^{N_\tau}u^{p^{N_\tau+1}} \varphi^{N_\tau+1}( S_{\operatorname{max}})$. In other words, $N_\nabla( \varphi(\iota_{\tau \circ \varphi})) = \varphi(\iota_{\tau \circ \varphi}) N^{(\tau)}$ with the entries of $N^{(\tau)}$ contained in the $\tau$-th part of $p^{N_\tau }u^{p^{N_\tau+1}}\varphi^{N_\tau+1}(S_{\operatorname{max}}) \otimes_{\mathbb{Z}_p} A$. Since the matrix $N^{(\tau)}$ is just the $\tau$-th part of the matrix $N$ defined above, the Taylor expansions of $N$ and $N_\tau$ around $u =\kappa(\pi)$ are equal whenever $\kappa|_{K_0}= \tau$. Therefore, to finish the proof one just has to show that if $f \in u^{p^{N_\tau+1}}\varphi^{N_\tau+1}(S_{\operatorname{max}}) \otimes_{\mathbb{Z}_p} A$ then the Taylor expansion $f =\sum_{i \geq 0} f_i (u-\kappa(\pi))^i$ around $u = \kappa(\pi)$ has $f_i \in \kappa(\pi)^{p^{N_\tau+1}-i}A$. To show this we use the formula $f_i = \frac{1}{i!} \left( \frac{d}{du} \right)^i(f)_{u=\kappa(\pi)}$
The fact that $f \in u^{p^{N_\tau+1}}\varphi^{N_\tau+1}(S_{\operatorname{max}}) \otimes A$ implies $f = \sum_{j \geq 1} u^{p^{N_\tau+1}j} \frac{f_j'}{\kappa(\pi)^{j-1}}$ with $f_j' \in A$. Thus
$$
f_i = \sum_{j \geq 1} \binom{p^{N_\tau+1} j }{i} \kappa(\pi)^{jp^{N_\tau+1} - i -j +1}f_j' 
$$
For $0 \leq i \leq p^{N_\tau+1}$ the $\kappa(\pi)$-adic valuation of $\binom{p^{N_\tau+1} j }{i} \kappa(\pi)^{jp^{N_\tau+1}  - i -j +1}$ is $\geq p^{N_\tau+1} -i$ whenever $j \geq 1$,\footnote{Notice this last estimate is weak when $i \geq 1$ because we ignore the potential divisibility of $\binom{p^{N_\tau+1} j }{i}$. However, when $i=0$ it is sharp and this it is this which critically limits our estimates of the $X_0^{-1}X_i$'s} and this finishes the proof.
\end{proof}

As an application of Proposition~\ref{prop-intxi} we can improve Corollary~\ref{cor-XMmu} to an integral assertion:
\begin{proposition}\label{prop-specialfibre}
	Additionally to the assumptions from Proposition~\ref{prop-intxi} suppose 
		\begin{itemize}
		\item the grading on $D_{K,\kappa}$ is concentrated in degrees $[-a_\kappa,0]$ whenever $\kappa|_{K_0} = \tau$ for $a_\kappa$ satisfying
		$$
		\sum_{\kappa|_{K_0} = \tau} a_\kappa \leq \frac{p^{N_\tau+1}-1}{\nu}+1,\qquad \nu = \operatorname{max} v_\pi(\pi-\pi')
		$$
		where $v_\pi$ denotes the $\pi$-adic valuation and the maximum runs over distinct conjugates $\pi'$ of $\pi$ in $\mathcal{O}_K$
	\end{itemize}
	Then, for any $(N_\tau)$-adapted basis $\iota$ of $\mathfrak{M}$ there exits $X_{\operatorname{trun}} \in \operatorname{Mat}(\mathfrak{m}_{\mathcal{O}}\mathfrak{S}_A)$ with $(1+ X_{\operatorname{trun}}) \cdot \Psi(\mathfrak{M},\iota) \in M_\mu(A)$.
\end{proposition}

\begin{proof}
	Let $\beta$ be a $K_0 \otimes_{\mathbb{Z}_p} A$-basis of $D$ and write $\xi(\varphi^*\iota) =\beta X$. For each $\kappa:K \hookrightarrow \mathcal{O}[\frac{1}{p}]$ write the Taylor expansion of $X$ around $u=\kappa(\pi)$ as $\sum_{i \geq 0} (u-\kappa(\pi))^i X_{\kappa,i}$. Since each $X_{\kappa,0} \in \operatorname{GL}_d(A[\frac{1}{p}])$ we can find $g \in \operatorname{GL}_d(K \otimes_{\mathbb{Z}_p} A) \cong \prod_{\kappa:K \hookrightarrow \mathcal{O}[\frac{1}{p}]} \operatorname{GL}_d(A[\frac{1}{p}])$ so that $X' :=g X$ has Taylor expansion
	$$
	\sum_{i \geq 0}(u-\kappa(\pi))^i X_{\kappa,0}^{-1} X_{\kappa,i}
	$$
	around $u =\kappa(\pi)$. Corollary~\ref{cor-XMmu} asserts that $X \cdot \Psi(\mathfrak{M},\iota) \in M_\mu[\frac{1}{p}])$. Since $M_{\mu}(A[\frac{1}{p}])$ is stable under the action of $ \operatorname{GL}_d(K \otimes_{\mathbb{Z}_p} A)$ it follows that $X' \cdot \Psi(\mathfrak{M},\iota) \in M_\mu[\frac{1}{p}]$ also. 
		
	Proposition~\ref{prop-intxi} implies that the Taylor expansions $\sum_{i \geq 0} (u-\kappa(\pi))^i X'_i$ are such that $X'_i \in \frac{p^{N_\tau}}{i}\kappa(\pi)^{p^{N_\tau+1}-i} A$ for $0 \leq i < p^{N_\tau+1}$. In particular, $X'_i \in \kappa(\pi)^{p^{N_\tau+1}-i}A$ for $0 \leq i < p^{N_\tau+1}$. Therefore, applying Lemma~\ref{lem-integralmatrices} below to the entries of $X'-1$ produces $X_{\operatorname{trun}} \in \operatorname{Mat}(\mathfrak{m}_{\mathcal{O}}\mathfrak{S}_A)$ so that $1+X_{\operatorname{trun}}$ has Taylor expansion around $u =\kappa(\pi)$ equal that of $X'$ modulo $\prod (u-\kappa(\pi))^{a_{\kappa}}$ for each $\kappa$. By Lemma~\ref{lem-truncate} this means $(1+X_{\operatorname{trun}})(X')^{-1}$ acts trivially on $M_\mu[\frac{1}{p}]$ and so $(1+X_{\operatorname{trun}}) \Psi(\mathfrak{M},\iota)[\frac{1}{p}] \in M_\mu(A[\frac{1}{p}])$. Since $1+X_{\operatorname{trun}} \in \operatorname{GL}_d(\mathfrak{S}_A)$ and $M_\mu$ is $\mathcal{O}$-flat by definition it follows that $(1+X_{\operatorname{trun}}) \Psi(\mathfrak{M},\iota) \in M_\mu(A)$ as desired.
\end{proof}

\begin{lemma}\label{lem-integralmatrices}
		Suppose that $f \in \widehat{\mathfrak{S}}_{A[\frac{1}{p}]}$ has Taylor expansions
		$$
		\sum f_{\kappa,i}(u-\kappa(\pi))^i
		$$
		around $u =\kappa(\pi)$ with $f_{\kappa,i} \in \kappa(\pi)^{p^{N_\tau+1}-n} A$ for $0 \leq i < p^{N_\tau+1}$. Suppose also that $a_\kappa \geq 0$ satisfy
		$$
		\sum_{\kappa|_{K_0} =\tau} a_{\kappa} \leq \frac{p^{N_\tau+1} -1}{\nu}+1
		$$
		for $\nu$ as in Proposition~\ref{prop-specialfibre}. Then there exists $f_{\operatorname{trun}} \in \mathfrak{m}_{\mathcal{O}}\mathfrak{S}_A$ with Taylor expansion at $u=\kappa(\pi)$ agreeing with that of $f$ modulo $(u-\kappa(\pi))^{a_\kappa}$ for each $\kappa$.
\end{lemma}
\begin{proof}
	The argument is identical to that in \cite[17.12]{B23}.
\end{proof}

If we consider the closed subfunctor $\widetilde{Z}_d^{(N_\tau)}$ of $\widetilde{Z}_d$ consisting of $(\mathfrak{M},\iota)$ with $\iota$ an $(N_\tau)$-adapted basis then Proposition~\ref{prop-specialfibre} implies the following:

\begin{corollary}\label{cor-factor}
	Let $\mu$ be a Hodge type and assume there are $N_\tau \geq 0$ are such 
	\begin{itemize}
		\item the grading on $D_{K,\kappa}$ is concentrated in degree $0$ whenever $\kappa|_{K_0} = \tau \circ \varphi^n$ for $1\leq n \leq N_\tau$
		\item the grading on $D_{K,\kappa}$ is concentrated in degrees $[-a_\kappa,0]$ whenever $\kappa|_{K_0} = \tau$ for $a_\kappa$ satisfying
		$$
		\sum_{\kappa|_{K_0} = \tau} a_\kappa \leq \frac{p^{N_\tau+1}-1}{\nu}+1,\qquad \nu = \operatorname{max} v_\pi(\pi-\pi')
		$$
		where $v_\pi$ denotes the $\pi$-adic valuation and the maximum runs over distinct conjugates $\pi'$ of $\pi$ in $\mathcal{O}_K$
	\end{itemize}
Then the composite $(Y^\mu \otimes_{\mathcal{O}} \mathbb{F}) \times_{Z_d} \widetilde{Z}^{(N_\tau)} \rightarrow \widetilde{Z}^{(N_\tau)} \xrightarrow{\Psi} \operatorname{Gr}$
factors through $M_\mu \otimes_{\mathcal{O}} \mathbb{F}$.
\end{corollary}
\begin{proof}
	Using \cite{B21} it suffices to prove the factorisation on $\overline{A}$-valued points for $\overline{A}$ a finite local $\mathbb{F}$-algebra. Let $(\overline{M},\overline{\iota})$ correspond to an $\overline{A}$-valued point of $(Y^\mu \otimes_{\mathcal{O}} \mathbb{F}) \times_{Z_d} \widetilde{Z}^{(N_\tau)}$. By \cite{B23} we can find a finite flat $\mathcal{O}$-algebra $A$ and $\mathfrak{M} \in Y^\mu_d(A)$ so that $A \otimes_{\mathcal{O}} \mathbb{F} = A$ and $\mathfrak{M} \otimes_{\mathcal{O}} \mathbb{F} =\overline{\mathfrak{M}}$. It is easy to see that $\overline{\iota}$ can also be lifted to an $(N_\tau)$-adapted basis $\iota$ of $\mathfrak{M}$. Proposition~\ref{prop-specialfibre} tells us that $(1+X_{\operatorname{trun}}) \Psi(\mathfrak{M},\iota) \in M_\mu(A)$ and, since $X_{\operatorname{trun}} \equiv 0$ modulo $\mathfrak{m}_{\mathcal{O}}$ it follows that, $\Psi(\overline{\mathfrak{M}},\overline{\iota}) = \Psi(\mathfrak{M},\iota) \otimes_{\mathcal{O}} \mathbb{F} \in M_\mu(\overline{A})$ as required.
\end{proof}

\section{Construction}

Here we prove part (1) from Theorem~\ref{thm-main}.  Notice that $\widetilde{Z}_d^{(N_\tau)}$ defined before Corollary~\ref{cor-factor} is closed in $\widetilde{Z}_d$ and is stabilised by two subgroups of $L^+\operatorname{GL}_d$ under the two actions from \eqref{eq-twoactions}. The stabiliser under the action $\cdot_{\operatorname{trans}}$ is the subgroup $L^+\operatorname{GL}_d^{(N_\tau),\operatorname{trans}}$ consisting of $g$ with $\tau \circ \varphi^n$-th part $g_{\tau \circ \varphi^n} = 1$ for $1\leq n \leq N_\tau$. The stabiliser under the action $\cdot _{\varphi}$ is the subgroup $L^+\operatorname{GL}_d^{(N_\tau),\varphi}$ consisting of $g \in L^+\operatorname{GL}_d$ with $g_{\tau \circ \varphi^n} = \varphi(g_{\tau \circ \varphi^{n+1}})$ whenever $1 \leq n \leq N_\tau$. The diagram from Section~\ref{sec-affinegrass} then restricts to a diagram
$$
\begin{tikzcd}
	& \ar[dl,"\Gamma"] \widetilde{Z}_d^{(N_\tau)} \ar[dr,"\Psi"] & \\
	Z_d & & \operatorname{Gr}
\end{tikzcd}
$$ 
in which $\Gamma$ and $\Psi$ are respectively torsors for the actions of $L^+\operatorname{GL}_d^{(N_\tau),\varphi}$ and $L^+\operatorname{GL}_d^{(N_\tau),\operatorname{trans}}$.
\begin{proposition}\label{prop-stability}
	Let be $\mu$ a Hodge type for which there are tuples of non-negative integers $(N_\tau)_{\tau:K_0\hookrightarrow\mathcal{O}[\frac{1}{p}]}$ such that the grading on the $\kappa$-th part of $\mu$:
	\begin{itemize}
		\item  is concentrated in degree $0$ whenever $\kappa|_{K_0} = \tau \circ \varphi^n$ with $1 \leq n \leq N_\tau$
		\item	is concentrated in degree $[-a_\kappa,0]$ whenever $\kappa|_{K_0} = \tau$ for $a_\kappa$ satisfying
		$$
		\sum_{\kappa|_{K_0} = \tau} a_\kappa \leq p^{N_\tau+1}
		$$
	\end{itemize} Then $\widetilde{Z}^{(N_\tau)}_d \times_{\operatorname{Gr}} (M_\mu \otimes_{\mathcal{O}} \mathbb{F})$ is stable under the action of $L^+\operatorname{GL}_d^{(N_\tau),\varphi}$, and so descends to a closed subfunctor $Z^\mu_{d,\mathbb{F}}$ of $Z_d$ characterised by the property from (1) of Theorem~\ref{thm-main}.
\end{proposition}

Notice that the assumptions on $\mu$ here are weaker then those imposed in Theorem~\ref{thm-main} since $p^{N_\tau+1} \geq \frac{p^{N_\tau+1}-1}{\nu} +1$ for $\nu \geq 1$.
\begin{proof}
	If $(\mathfrak{M},\iota) \in \widetilde{Z}_d$ then $ \Psi(g \cdot_{\varphi}(\mathfrak{M},\iota)) = \varphi(g) \cdot \Psi(\mathfrak{M},\iota)$. So we have to show that $M_\mu \otimes_{\mathcal{O}} \mathbb{F}$ is stable under the actions of $\varphi(L^+\operatorname{GL}_d^{(N_\tau),\varphi})$. The $\tau\circ \varphi^n$-th part of any element of $\varphi(L^+\operatorname{GL}_d^{(N_\tau),\varphi})$ is $\equiv g_{0,\tau \circ \varphi^n}$ modulo $u^{p^{N_\tau+1-n}}$ for $g_{0,\tau \circ \varphi^n} \in \operatorname{GL}_d(A)$ and $1\leq n \leq N_\tau$. Since $M_\mu(A)$ is evidently stable under $\operatorname{GL}_d(W(k) \otimes_{\mathbb{Z}_p} A)$ the proposition follows from the assertion that any $X \in L^+\operatorname{GL}_d$ with $\tau$-th part $\equiv 1$ modulo $\prod_{\kappa|_k =\tau} (u-\kappa(\pi))^{a_\kappa}$ acts trivially on $M_\mu$ when the grading on the $\kappa$-th part of $\mu$ is concentrated in degrees $[-a_{\kappa},0]$. This can be checked on the generic fibre where the assertion is just that of Lemma~\ref{lem-truncate}.
\end{proof}

Notice there is an obvious isomorphism between $L^+\operatorname{GL}_d^{(N_\tau),\varphi}$ and $L^+\operatorname{GL}_d^{(N_\tau),\operatorname{trans}}$. If we assumed that each of these group schemes was of finite type over $\mathcal{O}$ (which they are not) then it would follow that 
$$
\begin{aligned}
	 \operatorname{dim} Z^\mu_{d,\mathbb{F}} &= \operatorname{dim}\widetilde{Z}_d^{(N_\tau)} \times_{\operatorname{Gr}} (M_\mu \otimes_{\mathcal{O}} \mathbb{F}) + \operatorname{dim}L^+\operatorname{GL}_d^{(N_\tau),\varphi} \\ &= \operatorname{dim} \widetilde{Z}_d^{(N_\tau)} \times_{\operatorname{Gr}} (M_\mu \otimes_{\mathcal{O}} \mathbb{F}) + \operatorname{dim}L^+\operatorname{GL}_d^{(N_\tau),\operatorname{trans}}\\& = \operatorname{dim} M_\mu \otimes_{\mathcal{O}} \mathbb{F}
 \end{aligned}
 $$
 To make this argument work we instead descend the above diagram to ``finite level'' by considering quotients of $\widetilde{Z}_d$:
 
\begin{corollary}
	$\operatorname{dim} Z^\mu_{d,\mathbb{F}} = \operatorname{dim} M_\mu \otimes_{\mathcal{O}} \mathbb{F}$.
\end{corollary}
\begin{proof}
	For $M \geq 0$ let $U_M \subset L^+\operatorname{GL}_d$ denote the subgroup of matrices $\equiv 1$ modulo $u^M$. Set $U_{M}^{(N_\tau),\operatorname{trans}} = U_M \cap L^+\operatorname{GL}_d^{(N_\tau),\operatorname{trans}}$ and likewise for $U_M^{(N_\tau),\varphi}$. We claim there exists an $M\geq 0$ depending on $\mu$ so that 
	$$
	[\widetilde{Z}_d^{(N_\tau)} \times_{\operatorname{Gr}} (M_\mu \otimes_{\mathcal{O}} \mathbb{F})/U_{M}^{(N_\tau),\varphi}] \cong[ \widetilde{Z}_d^{(N_\tau)} \times_{\operatorname{Gr}} (M_\mu \otimes_{\mathcal{O}} \mathbb{F})/U_M^{(N_\tau),\operatorname{trans}}]
	$$
	with the quotients taken respectively for the $\cdot_\varphi$ and $\cdot_{\operatorname{trans}}$-action. Granting such an isomorphism we deduce
	$$
	\begin{aligned}
		\operatorname{dim} Z^\mu_{d,\mathbb{F}} &= \operatorname{dim} [\widetilde{Z}_d^{(N_\tau)} \times_{\operatorname{Gr}} (M_\mu \otimes_{\mathcal{O}} \mathbb{F})/U_{M}^{(N_\tau),\varphi}] - \operatorname{dim} L^+\operatorname{GL}_d^{(N_\tau),\varphi}/U_M^{(N_\tau),\varphi} \\
		&=\operatorname{dim} [\widetilde{Z}_d^{(N_\tau)} \times_{\operatorname{Gr}} (M_\mu \otimes_{\mathcal{O}} \mathbb{F})/U_{M}^{(N_\tau),\operatorname{trans}}] - \operatorname{dim} L^+\operatorname{GL}_d^{(N_\tau),\operatorname{trans}}/U_M^{(N_\tau),\operatorname{trans}}  \\
		&=   \operatorname{dim} M_\mu \otimes_{\mathcal{O}} \mathbb{F}	
	\end{aligned}
	$$
	where the second equality uses that the aforementioned isomorphism  $L^+\operatorname{GL}_d^{(N_\tau),\varphi} \cong L^+\operatorname{GL}_d^{(N_\tau),\operatorname{trans}}$ identifies $U_M^{(N_\tau),\varphi}$ and $U_M^{(N_\tau),\operatorname{trans}}$. To verify the isomorphism between the two quotients note that, after removing the $(N_\tau)$ superscripts this is just \cite[15.6]{B23}. The argument used in loc. cit. goes through identically when one only considers $(N_\tau)$-adapted bases.
\end{proof}

\section{Galois}
Here we finish the proof of (2) in Theorem~\ref{thm-main} by describing the fibres of  $Y_d^\mu \otimes_{\mathcal{O}} \mathbb{F} \rightarrow Z^\mu_{d,\mathbb{F}}$.

\begin{proposition}\label{prop-Galois}
	Let $A$ be an $\mathbb{F}$-algebra and $\mu$ a Hodge type for which there are tuples of non-negative integers $(N_\tau)_{\tau:K_0\hookrightarrow\mathcal{O}[\frac{1}{p}]}$ such that the grading on the $\kappa$-th part of $\mu$:
	\begin{itemize}
		\item  is concentrated in degree $0$ whenever $\kappa|_{K_0} = \tau \circ \varphi^n$ with $1 \leq n \leq N_\tau$
		\item	is concentrated in degree $[-a_\kappa,0]$ whenever $\kappa|_{K_0} = \tau$ for $a_\kappa$ satisfying
		$$
		\sum_{\kappa|_{K_0} = \tau} a_\kappa \leq (p^{N_\tau+1}-1)(1+\frac{e}{p-1})
		$$
		with the inequality strict for at least one $\tau:K_0 \hookrightarrow\mathcal{O}[\frac{1}{p}]$.
	\end{itemize}
Then any $\mathfrak{M}$ corresponding to an $A$-valued point of $Z^\mu_{d,\mathbb{F}}$ admits a unique continuous $\varphi$-equivariant $A_{\operatorname{inf},A}$-semilinear action of $G_K$ satisfying \eqref{eq-Galoisaction}.
\end{proposition}

The conditions imposed on $\mu$ here is again weaker than those in Proposition~\ref{prop-stability} or Theorem~\ref{thm-main} since $(p^{N_{\tau}+1}-1)(1+\frac{e}{p-1}) \geq p^{N_{\tau+1}} \geq \frac{p^{N_\tau}-1}{\nu}+1$.
\begin{proof}
	The idea is identical to that employed in \cite[\S11]{B21} or \cite[\S18]{B23}. Using the claimed uniqueness the action can be produced by glueing $G_K$-actions constructed on an open cover of $\operatorname{Spec}A$. Thus, we can assume that $\mathfrak{M}$ is $\mathfrak{S}_A$-free and so, by Lemma~\ref{lem-adaptedbasis} admits an $(N_\tau)$-adapted basis $\iota$. Then existence and uniqueness of a crystalline $G_K$-action is equivalent to the existence and uniqueness of a continuous cocycle
	$$
	c:G_K \rightarrow 1 + [\pi^\flat]\varphi^{-1}(\mu)\operatorname{Mat}(A_{\operatorname{inf},A}) = 1+ u^{1+\frac{e}{p-1}}\operatorname{Mat}(A_{\operatorname{inf,A}})
	$$
	satisfying $c(\sigma) \sigma(C) = C \varphi(c(\sigma))$ for $C$ defined by $\varphi(\iota) = \iota C$. The equality in the displayed equation above follows from the fact that $A$ is an $\mathbb{F}$-algebra and so the ideal of $A_{\operatorname{inf},A}$ generated by $\varphi^{-1}(\mu)$ equals that generated by $u^{e/(p-1)}$ (see \cite[5.1.3]{Fon94b}).
	
	For the rest of the proof we may, and do, assume that the $N_\tau$ are chosen as large as possible for the assumptions of the proposition to hold. This means that if $N_\tau >0$ then $N_\tau = N_{\tau \circ \varphi}+1$.
	
	\begin{observation}
		The assumptions on $\mu$ ensure the $\tau \circ \varphi$-th part of any such $c(\sigma)-1$ is divisible by $u^{p^{N_{\tau}}(1+\frac{e}{p-1})}$.
	\end{observation}
\begin{proof}[Proof of Observation]
	Since $\iota$ is $(N_\tau)$-adapted the $\tau \circ \varphi$-th part of $\iota$ can be written as $\varphi^{N_\tau}$ of elements in $\mathfrak{M}$. Thus, if $\sigma(m)$ denotes the semilinear action of $G_K$ on $\mathfrak{M} \otimes_{\mathfrak{S}_A} A_{\operatorname{inf},A}$ induced by the cocycle $c$, then $(\sigma-1)(\iota_{\tau \circ \varphi})$ is contained in the image of $\mathfrak{M} \otimes u^{1+\frac{e}{p-1}} A_{\operatorname{inf},A}$ under $\varphi^{N_\tau}$. We conclude that the $\tau \circ \varphi$-th part of $c(\sigma)-1$ has $u$-adic valuation $\geq p^{N_\tau}(1+\frac{e}{p-1})$ as claimed.
\end{proof}

	Motivated by this observation write $\mathcal{H} \subset \operatorname{Mat}(A_{\operatorname{inf},A})$ for the subgroup of matrices whose $\tau \circ \varphi$-th part is divisible by $u^{p^{N_\tau}(1+\frac{e}{p-1})}$. We claim that:
	\begin{enumerate}
		\item  $\mathcal{H}$ is stabilised by the operator $\Omega_\sigma$ defined by $M \mapsto C\varphi(M)\sigma(C^{-1})$
		\item  $\Omega_\sigma$ is topologically nilpotent on $\mathcal{H}$
		\item If $M \in \operatorname{Mat}(A_{\operatorname{inf},A})$ has entries divisible by $u^{1+\frac{e}{p-1}}$ then $\Omega_\sigma^n(M) \in \mathcal{H}$ for $n \geq \operatorname{max}_\tau{N_\tau}$. 
	\end{enumerate}To verify these points observe that the $\tau$-th part of $\Omega_\sigma(M)$ can be written as $C_{\tau} \varphi(M_{\tau\circ \varphi})\sigma(C_{\tau}^{-1})$
	where $C_\tau$ and $M_\tau$ denotes the $\tau$-th parts of $C$ and $M$. There are then two cases to consider. If $N_{\tau \circ \varphi^{-1}} >0$ then $C_\tau \in \operatorname{GL}_n(A[[u]])$ and so the $\tau\circ \varphi$-th part of $M$ being divisible by $u^x$ implies the $\tau$-th part of $\Omega_\sigma(M)$ has $u$-adic valuation $\geq px$. On the other hand, if $N_{\tau \circ \varphi^{-1}} =0$ then we know $u^{\sum_{\kappa|_{K_0}=\tau} a_\kappa}C^{-1}_\tau \in \operatorname{Mat}(A[[u]])$ and so the $\tau\circ \varphi$-th part of $M$ being divisible by $u^x$ implies the $\tau$-th part of $\Omega_\sigma(M)$ has $u$-adic valuation $\geq px -\sum_{\kappa|_{K_0} =\tau} a_\kappa$. If $x = p^{N_\tau}(1+\frac{e}{p-1})$ then in the first case $px = (p^{N_\tau+1})(1+\frac{e}{p-1}) = p^{N_\tau \circ \varphi^{-1}}(1+\frac{e}{p-1})$ and in the second $px =  p^{N_\tau+1}(1+\frac{e}{p-1}) - \sum_{\kappa|_{K_0} = \tau} a_{\kappa} \geq 1+\frac{e}{p-1}$. This proves the stability in (1). Since $\sum a_\kappa > (p^{N_\tau+1}-1)(1+\frac{e}{p-1})$ for at least one $\tau$ it also shows that topologically nilpotence in (2). (3) also follows easily from these observations.
	
	Claims (1) and (2) already suffice to prove uniqueness of the cocycle $c(\sigma)$. If $d(\sigma)$ is another such cocycle then $c(\sigma)-d(\sigma) \in \mathcal{H}$ and so
	$$
	0 =\lim_{n \rightarrow \infty} \Omega_\sigma^n(c(\sigma)-d(\sigma)) = c(\sigma)- d(\sigma)
	$$
	For existence, it follows from \cite[7.7]{B21} combined with the construction in \cite[11.2]{B21} that $(\mathfrak{M},\iota) \in \Psi^{-1}(M_\mu \otimes_{\mathcal{O}} \mathbb{F})$ ensures $C \sigma(C)^{-1}-1$ has entries divisible by $u^{1+\frac{e}{p-1}}$. Therefore (2) and (3) together ensure $\Omega_\sigma^n(C\sigma(C^{-1})-1)$ converges to zero in $\mathcal{H}$ as $n \rightarrow \infty$. Therefore $\Omega^n_\sigma(C\sigma(C^{-1}))$ is a Cauchy sequence and take $c(\sigma)$ as its limit. Since $\Omega_\sigma(c(\sigma)) = c(\sigma)$ we have $c(\sigma) \sigma(C) = C \varphi(c(\sigma))$. As $\sigma \mapsto C\sigma(C^{-1})$ is evidently a continuous cocycle the same is true of $\sigma \mapsto c(\sigma)$ which finishes the proof.
\end{proof}
\begin{corollary}\label{cor-closedimm}
	Assume that $\mu$ is as in Theorem~\ref{thm-main}. Then the factorisation from Corollary~\ref{cor-factor} induces a closed immersion $Y^\mu_d \otimes_{\mathcal{O}} \mathbb{F} \hookrightarrow Z^\mu_{d,\mathbb{F}}$ as in (2) of Theorem~\ref{thm-main}.
\end{corollary}\begin{proof}
Notice that  $\frac{p^{N_\tau+1}}{\nu}+ 1 < (p^{N_\tau+1} - 1)(1+\frac{e}{p-1})$
unless $N_\tau = 0$ and $\nu =1$. Therefore the assumptions on $\mu$ from Theorem~\ref{thm-main} allow Proposition~\ref{prop-Galois} to be applied. This allows $Z^{\mu}_{d,\mathbb{F}}$ to be identified with a closed substack of $Y_d$. Since the composite $Y^\mu_{d} \otimes_{\mathcal{O}} \mathbb{F} \rightarrow Z^\mu_{d,\mathbb{F}} \rightarrow Y^\mu_d$ is also a closed immersion we conclude $Y^\mu_d \otimes_{\mathcal{O}} \mathbb{F} \hookrightarrow Z^\mu_{d,\mathbb{F}}$ is also.
\end{proof}
\section{Irreducibility when $K = K_0$}

At this point parts (1) and (2) of Theorem~\ref{thm-main} have been proved. It remains only to establish (3). Namely, we need to prove:
\begin{proposition}\label{thm-isom}
	Assume $K = K_0$, that $K_\infty \cap K(\epsilon_\infty) = K$, and that $\mu$ is as in Theorem~\ref{thm-main}. Then the closed immersion from $Y^\mu_d \otimes_{\mathcal{O}} \mathbb{F} \hookrightarrow Z^\mu_{d,\mathbb{F}}$ is an isomorphism.
\end{proposition}
\begin{proof}
By \cite[16.7]{B23} the morphism $Y^\mu_d \otimes_{\mathcal{O}} \mathbb{F} \hookrightarrow Z^\mu_{d,\mathbb{F}}$ is a closed immersion between finite type algebraic stacks over $\mathbb{F}$ of the same dimension. This will necessarily be an isomorphism whenever $Z^\mu_{d,\mathbb{F}}$ is irreducible and reduced. Lemma~\ref{lem-MmuK=K0} says that $M_\mu \otimes_{\mathcal{O}} \mathbb{F}$ is irreducible and reduced. As the arrows in
$$
	Z_{d,\mathbb{F}}^\mu \leftarrow \widetilde{Z}_d^{(N_\tau)} \times_{\operatorname{Gr}} (M_\mu \otimes_{\mathcal{O}} \mathbb{F}) \rightarrow  M_\mu \otimes_{\mathcal{O}} \mathbb{F}
$$
are respectively $L^+\operatorname{GL}_d^{(N_\tau),\varphi}$ and $L^+\operatorname{GL}_d^{(N_\tau),\operatorname{trans}}$-torsors (both of which are smooth and irreducible) the irreducibility and reducedness of  $Z^\mu_{d,\mathbb{F}}$ from the analogous assertions for $M_\mu \otimes_{\mathcal{O}} \mathbb{F}$.
\end{proof}

	\bibliography{/home/user/Dropbox/Maths/biblio.bib}
\end{document}